\documentclass[12pt]{amsart}
\usepackage{graphicx}
\usepackage[headings]{fullpage}
\usepackage{amssymb,epic,eepic,epsfig,amsbsy,amsmath,amscd}
\numberwithin{equation}{section}
                        \theoremstyle{plain}
\usepackage{mathrsfs}

\newcommand{\psdraw}[2]
         {\begin{array}{c} \hspace{-1.3mm}
         \raisebox{-4pt}{\psfig{figure=#1.eps,width=#2}}
         \hspace{-1.9mm}\end{array}}
         
\newcommand\no[1]{}

\newtheorem{theorem}{Theorem}[section]
\newtheorem{thm}{Theorem}
\newtheorem{cor}{Corollary}

\newtheorem{lemma}[theorem]{Lemma}

\newtheorem{proposition}[theorem]{Proposition}

\theoremstyle{definition}

\def\BC{\mathbb C}

\def\BZ{\mathbb Z}

\def\BP{\mathbb P}

\def\fb{\mathfrak b}

\def\la{\langle}
\def\ra{\rangle}

\DeclareMathOperator{\tr}{\mathrm tr}

\def\ve{\varepsilon}
\def\be { \begin{equation} }
\def\ee { \end{equation} }

\begin{document}

\title[Character varieties of links]
{Character varieties of some families of links}

\author[Anh T. Tran]{Anh T. Tran}
\address{Department of Mathematics, The Ohio State University, Columbus, OH 43210, USA}
\email{tran.350@osu.edu}

\thanks{2010 {\em Mathematics Classification:} 57M27.}

\begin{abstract}
In this paper we consider some families of links, including $(-2,2m+1,2n)$-pretzel links and twisted Whitehead links. We calculate the character varieties of these families, and determine the number of irreducible components of these character varieties.
\end{abstract}

\maketitle

\setcounter{section}{-1}

\section{Introduction}

\subsection{The character variety of a group} The set of representations of a finitely generated group $G$ into $SL_2(\BC)$ is an algebraic set defined over $\BC$, on which
$SL_2(\BC)$ acts by conjugation. The set-theoretic
quotient of the representation space by that action does not
have good topological properties, because two representations with
the same character may belong to different orbits of that action. A better
quotient, the algebro-geometric quotient denoted by $X(G)$
(see \cite{CS, LM}), has the structure of an algebraic
set. There is a bijection between $X(G)$ and the set of all
characters of representations of $G$ into $SL_2(\BC)$, hence
$X(G)$ is usually called the character variety of $G$. 

The character variety of a group $G$ is determined by the traces of some fixed elements $g_1, \cdots, g_k$ in $G$. More precisely, one can find $g_1, \cdots, g_k$ in $G$ such that for every element $g$ in $G$ there exists a polynomial $P_g$ in $k$ variables such that for any representation $\rho: G \to SL_2(\BC)$ one has $\tr(\rho(g)) = P_g(x_1, \cdots, x_k)$ where $x_j:=\tr(\rho(g_j))$. It is known that the character variety of $G$ is equal to the zero set of the ideal of the polynomial ring $\BC[x_1, \cdots, x_k]$ generated by all expressions of the form $\tr(\rho(u))-\tr(\rho(v))$, where $u$ and $v$ are any two words in the letters $g_1, \cdots, g_k$ which are equal in $G$. 

\subsection{Main results} In this paper we consider some families of links, including $(-2,2m+1,2n)$-pretzel links and twisted Whitehead links. We will calculate the character varieties of these families,  and determine the number of irreducible components of these character varieties. To state our results, we first introduce the Chebyshev polynomials of the first kind $S_k(t)$. They are defined recursively by $S_0(t)=1$, $S_1(t)=t$ and $S_{k+1}(t)=t S_k(t)-S_{k-1}(t)$ for all integers $k$. 

The character variety (the character ring, actually) of the $(-2,2m+1,2n)$-pretzel link has been calculated in \cite{Tr}. Note that the $(-2,1,-2)$-pretzel link is the two-component unlink. Its link group is $\BZ^2$ and hence its character variety is $\BC^3$ by the  Fricke-Klein-Vogt theorem, see \cite{LM}. The new result of the following theorem is the determination of the number of irreducible components of the character variety.

\begin{thm}
\label{pretzel-link}
(i) \cite[Thm. 2]{Tr} The character variety of the $(-2,2m+1,2n)$-pretzel link is the zero set of the polynomial
$$(x^2+y^2+z^2-xyz-4) \left[ (xz-y)S_{n-1}(\alpha)-(S_{m}(\beta)-S_{m-1}(\beta))S_{n-2}(\alpha) \right],
$$
where 
$$
\alpha = yS_{m-1}(\beta)-(xz-y)S_{m-2}(\beta) \quad \text{and} \quad 
\beta = xyz+2-y^2-z^2. 
$$

(ii) The number of irreducible components of the character variety of the $(-2,2m+1,2n)$-pretzel link is equal to $
\begin{cases} 
|n+1| &\mbox{if } m=0 \mbox{ and } n \not= -1,\\
m+1 &\mbox{if } m \ge 0 \mbox{ and } n = 0, \\ 
-m& \mbox{if } m \le -1 \mbox{ and } n = 0, \\
2 &\mbox{if } m=1 \mbox{ and } n \not\in \{2,3\}, \\ 
3 & \mbox{if } m=1 \mbox{ and } n \in \{2,3\}, \\
|m|+1 & \mbox{if } n=-1, \\
2 &\mbox{if } m \not\in \{0,1\} \mbox{ and } n \not\in \{-1,0\}.
\end{cases} 
$
\end{thm}

\begin{figure}[htpb]
$$\psdraw{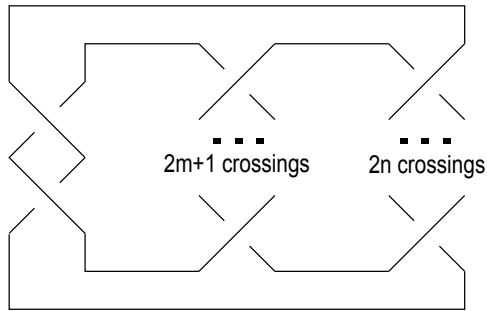}{2.5in}$$
\caption{The $(-2,2m+1,2n)$-pretzel link.}
\end{figure}

Let $\fb(2p,m)$ be the two-bridge link associated to a pair of relatively prime integers $p>m >0$, where $m$ is odd (see \cite{BZ}). It is known that $\fb(2p,1)$ is the $(2,2p)$-torus link, and $\fb(2p,m)$ is a hyperbolic link if $m \ge 3$. In the case $m=3$, we have the following.

\begin{thm}
\label{m=3}
The character variety of the two-bridge link $\fb(2p,3)$ is the zero set of the polynomial
$(x^2+y^2+z^2-xyz-4) Q_p(x,y,z)$, where $$
Q_p = \begin{cases} 
(x^2+y^2)S_n(z)S^2_{n-1}(z) - xyS_{n-1}(z)(S^2_n(z)+S^2_{n-1}(z))  + S_{3n}(z) &\mbox{if } p=3n+1 \\ 
(x^2+y^2)S^2_n(z)S_{n-1}(z) - xyS_{n}(z)(S^2_n(z)+S^2_{n-1}(z))  + S_{3n+1}(z) & \mbox{if } p=3n+2. 
\end{cases}
$$ It has exactly two irreducible components. 
\end{thm}

For two-bridge knots $\fb(p,3)$, where $p>3$ is an odd integer relatively prime with 3, a similar result about the number of irreducible components of their character varieties has been obtained in \cite{Bu, MPL, NT}.

The two-bridge link $\fb(6n+2,3)$ can be realized as $1/n$ Dehn filling on one cusp of the magic manifold, see e.g. \cite{La-thesis}. Note that $\fb(8,3)$ is the Whitehead link. In \cite{La} Landes identified the canonical component of the character variety of the Whitehead link as a rational surface isomorphic to $\BP^2$ blown up at 10 points. Here the canonical component of the character variety of a hyperbolic link is the component that contains the character of a discrete faithful representation, see \cite{Th}. In her thesis \cite{La-thesis}, Landes conjectured that the canonical component of the character variety of the two-bridge link $\fb(6n+2,3)$ is a rational surface isomorphic to $\BP^2$ blown up at $9n+1$ points. In a forthcoming paper \cite{PT}, we will confirm Landes' conjecture for all integers $n \ge 1$.

Similarly, Harada \cite{Ha} identified the canonical component of the character variety of the two-bridge link $\fb(10,3)$ as a rational surface isomorphic to $\BP^2$ blown up at 13 points. In \cite{PT}, we will also prove that the canonical component of the character variety of the two-bridge link $\fb(6n+4,3)$ is a rational surface isomorphic to $\BP^2$ blown up at $9n+4$ points, for all integers $n \ge 1$.

For $k \ge 0$, the $k$-twisted Whitehead link $W_k$ is the two-component link depicted in Figure 2. Note that $W_0$ is the $(2,4)$-torus link, and $W_1$ is the Whitehead link. Moreover, $W_k$ is the two-bridge link $\fb(4k+4,2k+1)$ for all $k \ge 0$. These links are all hyperbolic except for $W_0$. Their character varieties are described as follows.

\begin{figure}[htpb]
$$\psdraw{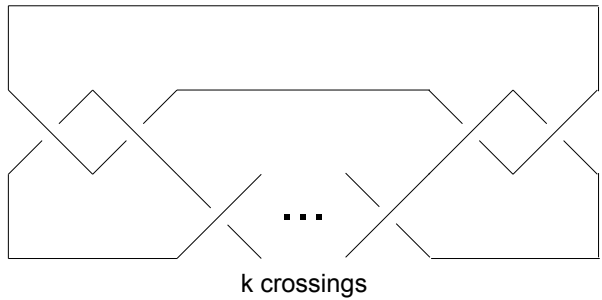}{3.1in} $$
\caption{The $k$-twisted Whitehead link $W_k$.}
\end{figure}

\begin{thm}
\label{twisted-Whitehead}
Let $\gamma=x^2+y^2+z^2-xyz-2$.

(i) The character variety of $W_{2n-1}$ is the zero set of the polynomial
$$(\gamma-2)\big( (xy-\gamma z)S_{n-1}(\gamma)-(xy-2z)S_{n-2}(\gamma) \big) \prod_{j=1}^{n-1}(\gamma-2\cos\frac{j\pi}{n}).$$
It has exactly $n+1$ irreducible components. 

(ii) The character variety of $W_{2n}$ is the zero set of the polynomial
$$(\gamma-2)\big( zS_{n}(\gamma)-(xy-z)S_{n-1}(\gamma) \big) \prod_{j=1}^{n}(\gamma-2\cos\frac{(2j-1)\pi}{2n+1}).$$
It has exactly $n+2$ irreducible components. 
\end{thm}


The twisted Whitehead link $W_{2n-1}$ can be realized as $1/n$ Dehn filling on one of the cusps of the Borromean rings. In \cite{Ha}, Harada identified the canonical component of the character variety of $W_2$, which is the two-bridge link $\fb(12,5)$, as a rational surface isomorphic to $\BP^2$ blown up at 10 points. Hence it is an interesting problem to understand the canonical component of the character variety of $W_k$ for all integers $k \ge 0$.
\no
{In \cite{La}, Landes proved that the non-abelian character variety of $W_{2n-1}$, for $2 \le n \le 4$, has a component which is a rational surface isomorphic to $\BP^2$ blown-up at 7 points. 

As a consequence of Theorem 3, we will show the following.

\begin{cor}
For $k \ge 2$, the non-abelian character variety of $W_k$ has at least $\lfloor \frac{k}{2} \rfloor$ components which are rational surfaces isomorphic to $\BP^2$ blown up at 7 points.
\end{cor}

Note that the components in Corollary 1 are not the canonical component. On the other hand,}

\subsection{Plan of the paper} In Section 1 we review some properties of the Chebyshev polynomials of the first kind. In Section 2 we recall the calculation of the character variety of the $(-2,2m+1,2n)$-pretzel link from \cite{Tr}, and prove the part of Theorem 1 on the determination of the number of irreducible components of its character variety. In Section 3 we review character varieties of two-bridge links and prove Theorems 2 and 3.

\subsection{Acknowledgment} We would like to thank K. Petersen for helpful discussions.

\section{Properties of Chebyshev polynomials}

Recall from the Introduction that the Chebyshev polynomials $S_k(t)$ are recursively defined by $S_0(t)=1$, $S_1(t)=t$ and $S_{k+1}(t)=t S_k(t)-S_{k-1}(t)$ for all integers $k$. 

In this section we list some properties of $S_k(t)$ which will be repeatedly used in the rest of the paper. 

\textbf{Property 1.1.}
One has $S_k(2)=k+1$ and $S_k(-2)=(-1)^k(k+1)$ for all integers $k$. Moreover if $t=q+q^{-1}$, where $q \not= \pm 1$, then $S_k(t)=\frac{q^{k+1}-q^{-k-1}}{q-q^{-1}}$.

\textbf{Property 1.2.} 
One has $S_{-k}(t)=-S_{k-2}(t)$ for all integers $k$. For $k \ge 0$, $$S_k(t)=\prod_{j=1}^k (t-2\cos \frac{j\pi}{k+1}) \quad \text{and} \quad S_k(t)-S_{k-1}(t)=\prod_{j=1}^k (t-2\cos \frac{(2j-1)\pi}{2k+1}).$$ 

\textbf{Property 1.3.}
One has $S^2_{k}(t)+S^2_{k-1}(t)-tS_k(t)S_{k-1}(t)=1$ for all integers $k$. As a consequence, $\gcd(S_k(t),S_{k-1}(t))=1$ in $\BC[t]$.

\textbf{Property 1.4.}
One has $S^3_{k}(t)-3S_{k}(t)S^2_{k-1}(t)+tS^3_{k-1}(t)=S_{3k}(t)$ for all integers $k$.

\medskip

The proof of Property 1.1 is elementary and hence is omitted. Properties 1.2--1.4 are proved by applying Property 1.1. We will only prove Property 1.4, and leave the proofs of Properties 1.2--1.3 for the reader.

It is easy to see that we only need to check Property 1.4 for $t \not= \pm 2$. Write $t=q+q^{-1}$ for some $q \not= \pm 1$. Then, by Property 1.1, we have
\begin{eqnarray*}
&& S^3_{k}(t)-3S_{k}(t)S^2_{k-1}(t)+tS^3_{k-1}(t) \\
&=& \left( \frac{q^{k+1}-q^{-k-1}}{q-q^{-1}} \right)^3 - 3 \left( \frac{q^{k+1}-q^{-k-1}}{q-q^{-1}} \right) \left( \frac{q^{k}-q^{-k}}{q-q^{-1}} \right)^2+(q+q^{-1}) \left( \frac{q^{k}-q^{-k}}{q-q^{-1}} \right)^3 \\
&=& \frac{q^{3k+1}-q^{-3k-1}}{q-q^{-1}} = S_{3k}(t).
\end{eqnarray*}
This completes the proof of Property 1.4.

\section{$(-2,2m+1,2n)$-pretzel links}

In this section, we let $L$ denote the $(-2,2m+1,2n)$-pretzel link in Figure 1. From \cite[Thm. 2]{Tr}, we have that the character ring of $L$ is the quotient of the polynomial ring $\BC[x,y,z]$ by the principal ideal generated by the polynomial
$$(x^2+y^2+z^2-xyz-4) \left[ (xz-y)S_{n-1}(\alpha)-(S_{m}(\beta)-S_{m-1}(\beta))S_{n-2}(\alpha) \right],
$$
where $\alpha = yS_{m-1}(\beta)-(xz-y)S_{m-2}(\beta)$ and $\beta = xyz+2-y^2-z^2.$

Let $$Q(x,y,z)=(xz-y)S_{n-1}(\alpha)-(S_{m}(\beta)-S_{m-1}(\beta))S_{n-2}(\alpha).$$ Then the character variety of $L$ is the zero set of $(x^2+y^2+z^2-xyz-4)Q(x,y,z)$. 

To determine the number of components of the character variety of $L$, we need to study the factorization of $Q(x,y,z)$ in $\BC[x,y,z]$.

\subsection{The case $m=0$} Then $\alpha = xz-y$ and hence $$Q=(xz-y)S_{n-1}(xz-y)-S_{n-2}(xz-y)=S_n(xz-y).$$

If $n \ge 0$ then $Q = S_n(xz-y) = \prod_{j=1}^n (xz-y-2\cos \frac{j\pi}{n+1}),$ by Property 1.2. Similarly, if $n \le -2$ then $Q = -S_{-(n+2)}(xz-y) = -\prod_{j=1}^{-(n+2)} (xz-y-2\cos \frac{j\pi}{-(n+1)}).$

In this case we have the following.

\begin{proposition} 
\label{prop.first}
The number of components of the non-abelian character variety of the $(-2,1,2n)$-pretzel link, where $n \not= -1$, is equal to $
\begin{cases} n &\mbox{if } n \ge 0, \\ 
-(n+2)& \mbox{if } n \le -2.\end{cases} 
$
\end{proposition}

\subsection{The case $m=1$} Then $\alpha =y$ and hence
\begin{eqnarray*}
Q &=& (xz-y)S_{n-1}(y)-(xyz+1-y^2-z^2)S_{n-2}(y)\\
  &=& -xzS_{n-3}(y)+z^2S_{n-2}(y)+S_{n-4}(y).
\end{eqnarray*} 

If $n=2$ then $Q=(z-1)(z+1)$. If $n=3$ then $Q=z(yz-x)$. Suppose now that $n \notin \{2,3\}$. Since $zS_{n-3}(y) \not\equiv 0$, $Q$ has degree 1 in $x$. This, together with the fact that $\gcd(zS_{n-3}(y), z^2S_{n-2}(y)+S_{n-4}(y))=1$, implies that $Q$ is irreducible in $\BC[x,y,z].$ 

In this case we have the following.

\begin{proposition}
The number of irreducible components of the non-abelian character variety of the $(-2,3,2n)$-pretzel knot is equal to 
$
\begin{cases} 1 &\mbox{if } n \not\in \{2,3\}, \\ 
2 & \mbox{if } n \in \{2,3\} .\end{cases} 
$
\end{proposition}

\subsection{The case $n=0$} Then $Q=S_{m}(\beta)-S_{m-1}(\beta)$, where $\beta = xyz+2-y^2-z^2.$ 

If $m \ge 0$ then $Q = \prod_{j=1}^m (xyz+2-y^2-z^2-2\cos \frac{(2j-1)\pi}{2m+1})$, by Property 1.2. Similarly, if $m \le -1$ then $Q = S_{-m-1}(\beta)-S_{-m-2}(\beta) = \prod_{j=1}^{-(m+1)} (xyz+2-y^2-z^2-2\cos \frac{(2j-1)\pi}{-(2m+1)}).$ 

Note that, for any $\delta \in \BC$, the polynomial $xyz-y^2-z^2+ \delta$ is irreducible in $\BC[x,y,z]$. Hence, in this case we have the following.

\begin{proposition} The number of irreducible components of the non-abelian character variety of the $(-2,2m+1,0)$-pretzel link is equal to $
\begin{cases} m &\mbox{if } m \ge 0, \\ 
-(m+1)& \mbox{if } m \le -1.\end{cases} 
$
\end{proposition}

\subsection{The case $n=-1$} Then
\begin{eqnarray*}
Q &=& -(xz-y)+(S_{m}(\beta)-S_{m-1}(\beta))(yS_{m-1}(\beta)-(xz-y)S_{m-2}(\beta)) \\
         &=& (y-xz)  \left[ 1+(S_{m}(\beta)-S_{m-1}(\beta))S_{m-2}(\beta) \right]+y(S_{m}(\beta)-S_{m-1}(\beta))S_{m-1}(\beta),
\end{eqnarray*}
where $\beta = xyz+2-y^2-z^2$. By Property 1.3, we have $$1+S_{m}(\beta)S_{m-2}(\beta)=1+S_{m}(\beta) (tS_{m-1}(\beta)-S_m(\beta))=S^2_{m-1}(\beta).$$ It follows that $1+(S_{m}(\beta)-S_{m-1}(\beta)) S_{m-2}(\beta)=S_{m-1}(\beta) (S_{m-1}(\beta) - S_{m-2}(\beta))$. Hence
$$Q=S_{m-1}(\beta) \left[ y(S_{m}(\beta)-S_{m-1}(\beta)) -(xz-y)(S_{m-1}(\beta) - S_{m-2}(\beta))\right].$$

In this case, we will prove the following.

\begin{proposition}
\label{n=-1}
The number of irreducible components of the non-abelian character variety of the $(-2,2m+1,-2)$-pretzel link is equal to $|m|$ if $m \not= 0$.
\end{proposition}

This is equivalent to showing that
$$R(x,y,z) := y(S_{m}(\beta)-S_{m-1}(\beta)) - (xz-y)(S_{m-1}(\beta) - S_{m-2}(\beta))$$ is a non-constant irreducible polynomial in $\BC[x,y,z]$. 

We first prove that $R$ is non-constant. Indeed, we have $R \mid_{z=0} \, = y(S_m(\gamma)-S_{m-2}(\gamma))$ where $\gamma =2-y^2$. It follows that $R \mid_{y= \pm 2, z=0} \, = \pm 4(-1)^m$. Hence $R$ is non-constant. Moreover we have $\gcd(z,R)=1$, since $R \mid_{z=0} \, \not\equiv 0.$ 

We now prove that $R$ is irreducible in $\BC[x,y,z]$. Recall that $\beta = (xz-y)y+2-z^2$ and $R = y(S_{m}(\beta)-S_{m-1}(\beta))-(xz-y)(S_{m-1}(\beta) - S_{m-2}(\beta))$. Note that $x=\frac{(xz-y)+y}{z}$ (for $z \not=0$) and $\gcd(z,R)=1$. 

Hence to prove the irreducibility of $R$ in $\BC[x,y,z]$, we only need to prove that
$$R_1(x_1,y,z):=y(S_{m}(\beta_1)-S_{m-1}(\beta_1))-x_1(S_{m-1}(\beta_1) - S_{m-2}(\beta_1)),$$
where $\beta_1=x_1 y+2-z^2$, is irreducible in $\BC[x_1,y,z]$. 

\medskip

\textbf{Claim 1.} \textit{$R_1$ is irreducible in $\BC[x_1,y,z^2]$.}

\medskip

Since $z^2=x_1y+2-\beta_1$, proving Claim 1 is equivalent to proving that 
$$R_2(x_1,y,\beta_2):=y(S_{m}(\beta_2)-S_{m-1}(\beta_2))-x_1(S_{m-1}(\beta_2) - S_{m-2}(\beta_2)),$$
is irreducible in $\BC[x_1,y,\beta_2]$. 

Since $S_{m}(\beta_2)-S_{m-1}(\beta_2) \not\equiv 0$, $R_2$ has degree 1 in $y$. This, together with the fact that $\gcd(S_{m}(\beta_2)-S_{m-1}(\beta_2),S_{m-1}(\beta_2) - S_{m-2}(\beta_2))=1$, implies that $R_2$ is irreducible in $\BC[x_1,y,\beta_2]$. Claim 1 follows.

\medskip

\textbf{Claim 2.} \textit{$R_1$ is irreducible in $\BC[x_1,y,z]$.}

\medskip

Assume that $R_1$ is reducible in $\BC[x_1,y,z]$. Since $R_1$ is irreducible in $\BC[x_1,y,z^2]$, it must have the form $R_1=(fz+g)(-fz+g)$, for some $f,g \in \BC[x_1,y,z^2] \setminus \{0\}$ satisfying $(f,g)=1$. In particular, $R_1 \mid_{z=0}$ is a perfect square in $\BC[x_1,y]$.  This can not occur, since $R_1 \mid_{x_1=0, z=0} \, =y$ is not a perfect square in $\BC[y]$. Claim 2 follows.

\medskip

This completes the proof of Proposition \ref{n=-1}.

\subsection{The case $m \not \in \{0,1\} $ and $n \not\in \{-1,0\}$} In this case, we will prove the following.

\begin{proposition}
\label{prop.last}
The non-abelian character variety of the $(-2,2m+1,2n)$-pretzel link, where  $m \not \in \{0,1\} $ and $n \not\in \{-1,0\}$, is irreducible.
\end{proposition}

This is equivalent to showing that $Q(x,y,z)$ is a non-constant irreducible polynomial in $\BC[x,y,z]$.

\medskip

We first prove that $Q$ is non-constant. Indeed, assume that $Q(x,y,z)$ is a constant polynomial. By \cite[Prop. 2.5]{Tr}, we have $Q \mid_{z=0} \, =(-1)^{(m-1)(n-1)} S_{2mn-2m-n-2}(y)$. Hence $S_{2mn-2m-n-2}(y)$ is also a constant polynomial. It follows that $2mn-2m-n-2 \in \{-2,-1,0\}$. Since $m \not \in \{0,1\} $ and $n \not\in \{-1,0\}$, we must have $(m,n)=(2,2)$. However, in the case $(m,n)=(2,2)$ we have $Q=-x^2 z^2+x y z^3+x y z-y^2 z^2-z^4+3 z^2-1$. Hence $Q$ is non-constant. Moreover we have $\gcd(z,Q)=1$, since $Q \mid_{z=0} \not\equiv 0$. 

\medskip

We now prove that $Q$ is irreducible in $\BC[x,y,z]$. Recall that $\alpha = yS_{m-1}(\beta)-(xz-y)S_{m-2}(\beta)$, $\beta = (xz-y)y+2-z^2$ and 
$$Q=(xz-y)S_{n-1}(\alpha)-(S_{m}(\beta)-S_{m-1}(\beta))S_{n-2}(\alpha).$$

Note that $x=\frac{(xz-y)+y}{z}$ (for $z \not=0$) and $\gcd(z,Q)=1$. Hence to prove the irreducibility of $Q$ in $\BC[x,y,z]$, we only need to prove that
$$Q_1(x_1,y,z):=x_1 S_{n-1}(\alpha_1)-(S_{m}(\beta_1)-S_{m-1}(\beta_1))S_{n-2}(\alpha_1),$$
where $\alpha_1 = yS_{m-1}(\beta_1)-x_1 S_{m-2}(\beta_1)$ and $\beta_1=x_1 y+2-z^2$, is irreducible in $\BC[x_1,y,z]$. 

\medskip

\textbf{Claim 3.} \textit{$Q_1$ is irreducible in $\BC[x_1,y,z^2]$.}

\medskip

Since $z^2=x_1y+2-\beta_1$, proving Claim 3 is equivalent to proving that 
$$Q_2(x_1,y,\beta_2):=x_1 S_{n-1}(\alpha_2)-(S_{m}(\beta_2)-S_{m-1}(\beta_2))S_{n-2}(\alpha_2),$$
where $\alpha_2 = yS_{m-1}(\beta_2)-x_1 S_{m-2}(\beta_2)$, is irreducible in $\BC[x_1,y,\beta_2]$. 

\medskip

The proof of the irreducibility of $Q_2$ in $\BC[x_1,y,\beta_2]$ is divided into 2 steps.

\medskip

\textit{Step 1.} We first show that $\gcd(S_{m-2}(\beta_2),Q_2)=1$. Indeed, since $m \not= 1$ we have $S_{m-2}(\beta_2) \not\equiv 0$. Suppose that $S_{m-2}(\beta_2)=0$. By Property 1.3, we have $S^2_{m-1}(\beta_2)+S^2_{m-2}(\beta_2)-\beta_2 S_{m-1}(\beta_2)S_{m-2}(\beta_2)=1$. It follows that $S_{m-1}(\beta_2)=\ve$ for some $\ve \in \{\pm 1\}$. Hence $\alpha_2 = \ve y$ and $Q_2  \mid_{S_{m-2}(\beta_2)=0} \, =x_1 S_{n-1}(\ve y)-(\beta_2 -1)\ve S_{n-2}(\ve y)$. Since $n \not= 0$, we have $S_{n-1}(\ve y) \not\equiv 0$. Therefore $Q_2 \mid_{S_{m-2}(\beta_2)=0} \, \not\equiv 0$, which implies that $\gcd(S_{m-2}(\beta_2),Q_2)=1$. 

\medskip

\textit{Step 2.} For $S_{m-2}(\beta_2) \not=0$ we have $x_1=\frac{yS_{m-1}(\beta_2) - \alpha_2}{S_{m-2}(\beta_2)}.$ Since $\gcd(S_{m-2}(\beta_2),Q_2)=1$, to prove the irreducibility of $Q_2$ in $\BC[x_1,y,\beta_2]$ we only need to prove that
$$Q_3(x_2,y,\beta_2):=(yS_{m-1}(\beta_2) - x_2) S_{n-1}(x_2)-S_{m-2}(\beta_2)(S_{m}(\beta_1)-S_{m-1}(\beta_1))S_{n-2}(x_2)$$
is irreducible in $\BC[x_2,y,\beta_2]$. 

Since $mn \not=0$ we have $S_{m-1}(\beta_2)S_{n-1}(x_2) \not\equiv 0$. It follows that $Q_3$ has degree 1 in $y$. We write $Q_3=yf-g$, where $f=S_{m-1}(\beta_2)S_{n-1}(x_2)$ and $g=x_2 S_{n-1}(x_2)+S_{m-2}(\beta_2)(S_{m}(\beta_2)-S_{m-1}(\beta_2)) S_{n-2}(x_2).$ Then $Q_3$ is irreducible in $\BC[x_2, y, \beta_2]$ if $\gcd(f,g)=1$. 

Since $m \not= 1$, we have $S_{m-2}(\beta_2) (S_{m}(\beta_2)-S_{m-1}(\beta_2)) \not\equiv 0$. This, together with the fact that $\gcd(S_{n-1}(x_2),S_{n-2}(x_2))=1$, implies that $\gcd(S_{n-1}(x_2),g)=1.$

By Property 1.3, we have $$S_{m}(\beta_2)S_{m-2}(\beta_2)=S_{m}(\beta_2) (tS_{m-1}(\beta_2)-S_m(\beta_2))=S^2_{m-1}(\beta_2)-1.$$ It follows that $$\gcd(S_{m-1}(\beta_2),g)=\gcd(S_{m-1}(\beta_2),x_2 S_{n-1}(x_2)-S_{n-2}(x_2))=\gcd(S_{m-1}(\beta_2),S_{n}(x_2)).$$
Since $n \not= -1$, we have $S_{n}(x_2) \not\equiv 0$ and hence $\gcd(S_{m-1}(\beta_2),g)=1$. Therefore $\gcd(f,g)=1$, which implies the irreducibility of $Q_3$ in $\BC[x_2, y, \beta_2]$. Claim 3 follows.

\medskip

\textbf{Claim 4.} \textit{$Q_1$ is irreducible in $\BC[x_1,y,z]$.}

\medskip

Assume that $Q_1$ is reducible in $\BC[x_1,y,z]$. Since $Q_1$ is irreducible in $\BC[x_1,y,z^2]$, it must have the form $Q_1=(fz+g)(-fz+g)$, for some $f,g \in \BC[x_1,y,z^2] \setminus \{0\}$ satisfying $(f,g)=1$. In particular, $Q_1 \mid_{z=0}$ is a perfect square in $\BC[x_1,y]$. 

Since $m \not=0$ and $n \not\in \{0,1\}$, $Q_1 \mid_{x_1=0, z=0} \, =-S_{n-2}(my)$ is not a perfect square in $\BC[y]$ unless $n=2$. In the case $n=2$ we have $Q_1 \mid_{y=0, z=0}=(1-m)x_1^2-1$ is not a perfect square in $\BC[x_1]$, since $m \not= 1$. This shows that $Q_1 \mid_{z=0}$ can not be a perfect square in $\BC[x_1,y]$. Claim 4 follows.

\medskip

This completes the proof of Proposition \ref{prop.last}.

\subsection{Proof of Theorem 1(ii)} If $(m,n) \not= (0,-1)$ then Theorem 1(ii) follows from Propositions \ref{prop.first}--\ref{prop.last}. If $(m,n)=(0,-1)$ then $L$ is the $(-2,1,-2)$-pretzel link, which is the two-component unlink. Its link group is $\BZ^2$ and hence its character variety is $\BC^3$ by the  Fricke-Klein-Vogt theorem, see \cite{LM}. This completes the proof of Theorem 1(ii).

\section{Two-bridge links}

We first review character varieties of two-bridge links. Let $L=\fb(2p,m)$ be the two-bridge link associated to a pair of relatively prime integers $p>m >0$, where $m$ is odd (see \cite{BZ}). The link group of $L$ is $\pi_L=\la a,b \mid aw=wa\ra$, where $w=b^{\varepsilon_1}a^{\varepsilon_2} \cdots a^{\varepsilon_{2p-2}}b^{\varepsilon_{2p-1}}$ and $\varepsilon_j=(-1)^{\lfloor \frac{mj}{2p} \rfloor}.$ Here $a$ and $b$ are 2 meridians of $L$.

Let $F_{a,b}:=\la a,b \ra$ be the free group in 2 letters $a$ and $b$. The character variety of $F_{a,b}$ is isomorphic to $\BC^3$ by the Fricke-Klein-Vogt theorem. For every word $u$ in $F_{a,b}$ there is a \emph{unique} polynomial $P_u$ in 3 variables such that for any representation $\rho: F_{a,b} \to SL_2(\BC)$ one has $\tr (\rho(u))=P_u (x,y,z)$ where $x:=\tr(\rho(a)),~y:=\tr(\rho(b))$ and $z:=\tr(\rho(ab))$. By \cite[Prop. 1.4.1]{CS}, the polynomial $P_u$ can be calculated inductively using the following identity for traces of matrices $A,B \in SL_2(\BC)$:
\begin{equation}
\label{trace}
\tr (AB)+\tr(AB^{-1})=\tr(A) \tr(B).
\end{equation}

Suppose $G$ be a group generated by 2 elements $a$ and $b$. For every representation $\rho: G \to SL_2(\BC)$, we consider $x,y,$ and $z$ as functions of $\rho$. By abuse of notation, we will identify $u \in G$ with its image $\rho(u) \in SL_2(\BC)$.

By \cite{LT}, the character ring of $\pi_L=\la a,b \mid aw=wa\ra$ is the quotient of the polynomial ring $\BC[x,y,z]$ by the principal ideal generated by the polynomial $P_{awa^{-1}b^{-1}}-P_{wb^{-1}}$. It follows that the character variety of $L$ is the zero set of $P_{awa^{-1}b^{-1}}-P_{wb^{-1}}$. 

The following lemma will be frequently used in the proofs of Theorems 2 and 3.

\begin{lemma} 
\label{recurrence}
Suppose $M \in SL_2(\BC)$. Then, for all integers $k$, 
\begin{equation}
\label{matrix}
M^k = S_{k}(\tr M)I - S_{k-1}(\tr M)M^{-1}.
\end{equation}
\end{lemma}

\begin{proof}
For $M \in SL_2(\BC)$, the Caley-Hamilton theorem implies that $M=(\tr M)I-M^{-1}$, where $I$ denotes the identity matrix in $SL_2(\BC)$. By induction on $k$, we can show that Eq. \eqref{matrix} holds true for all $M \in SL_2(\BC)$ and $k \ge 0$.

For $k<0$, by applying Eq. \eqref{matrix} for $M^{-1} \in SL_2(\BC)$ and $-k>0$ we have 
\begin{eqnarray*}
(M^{-1})^{-k} &=& S_{-k}(\tr M^{-1})I - S_{-k-1}(\tr M^{-1})M \\
              &=& -S_{k-2}(\tr M)I + S_{k-1}(\tr M) \left( (\tr M)I-M^{-1} \right)\\
              &=& S_{k}(\tr M)I - S_{k-1}(\tr M)M^{-1}.
\end{eqnarray*}
The lemma follows.
\end{proof}

We will also need the following.

\begin{lemma}
\label{gamma}
One has $P_{aba^{-1}b^{-1}}=x^2+y^2+z^2-xyz-2$.
\end{lemma}

\begin{proof}
By applying Eq. \eqref{trace}, we have $P_{(ab)(ba)^{-1}} = P_{ab} P_{ba}-P_{ab^2a} = z^2-(P_{ab^2}P_a-P_{b^2})$. 

Similarly, $P_{ab^2}=P_{ab}P_b-P_{a}=yz-x$ and $P_{b^2}=y^2-2$. Hence
$P_{(ab)(ba)^{-1}}=z^2+P_{b^2}-xP_{ab^2}=x^2+y^2+z^2-xyz-2.$
\end{proof}

We now prove Theorems 2 and 3. Let $\gamma=P_{aba^{-1}b^{-1}}$. Then, by Lemma \ref{gamma}, we have $\gamma=x^2+y^2+z^2-xyz-2$.

\subsection{Two-bridge links $\fb(2p,3)$} The link group of $L=\fb(2p,3)$ is $\pi_L=\la a,b \mid aw=wa\ra$ where $w=b^{\varepsilon_1}a^{\varepsilon_2} \cdots a^{\varepsilon_{2p-2}}b^{\varepsilon_{2p-1}}$, and $\ve_j=-1$ if $\lfloor \frac{2p}{3} \rfloor +1 \le j \le \lfloor \frac{4p}{3} \rfloor$ and $\ve_j=1$ otherwise.

The proof of Theorem 2 for the two-bridge link $\fb(6n+4,3)$ will be similar to that for $\fb(6n+2,3)$. Hence, without loss of generality we can assume that $L=\fb(6n+2,3)$. In this case, we have $p=3n+1$ and $w=(ba)^{n}(b^{-1}a^{-1})^{n}b^{-1}(ab)^{n}.$ 

We first calculate the character variety of $L$. Recall that $x=P_a$, $y=P_b$ and $z=P_{ab}$. In $SL_2(\BC)$, by applying Eq. \eqref{matrix} we have
\begin{eqnarray*}
a^{-1}wab^{-1} &=& a^{-1}(ba)^{n}(b^{-1}a^{-1})^{n}b^{-1}(ab)^{n}ab^{-1}\\
               &=& a^{-1} \left( S_{n}(z)I-S_{n-1}(z)(ba)^{-1} \right) \left( S_{n}(z)I-S_{n-1}(z)(b^{-1}a^{-1})^{-1} \right) \\
               && b^{-1} \left( S_{n}(z)I-S_{n-1}(z)(ab)^{-1} \right) ab^{-1}\\
                   &=& a^{-1}b^{-1}ab^{-1} S^3_{n}(z) - (a^{-2}b^{-2}ab^{-1} +  ab^{-1} + a^{-1}b^{-3} ) S^2_{n}(z)S_{n-1}(z) \\
                     && + ( a^{-2}b^{-1}a^2b^{-1} + a^{-2}b^{-4}  + b^{-2} ) S_{n}(z)S^2_{n-1}(z) - a^{-2} b^{-1} ab^{-2} S^3_{n-1}(z).
\end{eqnarray*}
Similarly,
\begin{eqnarray*}
wb^{-1} &=& (ba)^{n}(b^{-1}a^{-1})^{n}b^{-1}(ab)^{n}b^{-1}\\
                   &=& b^{-2} S^3_{n}(z) - (a^{-1}b^{-3} + ab^{-1} + b^{-2}a^{-1}b^{-1} ) S^2_{n}(z)S_{n-1}(z) \\
                     && + ( a^{-1}b^{-1}ab^{-1} + a^{-1}b^{-3}a^{-1}b^{-1} + ab^{-1}a^{-1}b^{-1} ) S_{n}(z)S^2_{n-1}(z) \\
                     && - a^{-1} b^{-1} ab^{-1}a^{-1}b^{-1} S^3_{n-1}(z).
\end{eqnarray*}

Hence
\begin{eqnarray}
\label{matrix-simplify}
                     && a^{-1}wab^{-1} - wb^{-1} \\
                     &=& (a^{-1}b^{-1}ab^{-1} - b^{-2}) S^3_{n}(z) - (a^{-2}b^{-2}ab^{-1} - b^{-2}a^{-1}b^{-1}) S^2_{n}(z)S_{n-1}(z) \nonumber\\
                     && + ( a^{-2}b^{-1}a^2b^{-1} + a^{-2}b^{-4}  + b^{-2} - a^{-1}b^{-1}ab^{-1} - a^{-1}b^{-3}a^{-1}b^{-1} - ab^{-1}a^{-1}b^{-1}) \nonumber \\ && S_{n}(z)S^2_{n-1}(z) - (a^{-2} b^{-1} ab^{-2} - a^{-1} b^{-1} ab^{-1}a^{-1}b^{-1}) S^3_{n-1}(z). \nonumber
\end{eqnarray}

By applying the algorithm for calculating $P_u$, for any word $u$ in 2 letters $a$ and $b$, described in \cite[prop. 1.4.1]{CS}, we can prove the following. 

\begin{lemma} 
\label{trace-simplify}
One has
\begin{eqnarray*}
P_{a^{-1}b^{-1}ab^{-1}} - P_{b^{-2}} &=& 2-\gamma,\\
P_{a^{-2}b^{-2}ab^{-1}} - P_{b^{-2}a^{-1}b^{-1}} &=& (2-\gamma)xy,\\
P_{a^{-2}b^{-1}a^2b^{-1}} - P_{a^{-1}b^{-1}ab^{-1}} &=& (2-\gamma)(x^2-1),\\
P_{a^{-2}b^{-4}} - P_{a^{-1}b^{-3}a^{-1}b^{-1}} &=& (2-\gamma)(y^2-1),\\
P_{b^{-2} } - P_{ab^{-1}a^{-1}b^{-1}} &=& \gamma-2,\\
P_{a^{-2} b^{-1} ab^{-2}} - P_{a^{-1} b^{-1} ab^{-1}a^{-1}b^{-1}} &=& (2-\gamma) (xy-z).
\end{eqnarray*}
\end{lemma}

Eq. \eqref{matrix-simplify} and Lemma \ref{trace-simplify} imply that
\begin{eqnarray*}
P_{a^{-1}wab^{-1}}-P_{wb^{-1}} &=& (2-\gamma) \big[  (x^2+y^2)S_{n}(z)S^2_{n-1}(z)-xyS_{n-1}(z)(S^2_{n}(z)+S^2_{n-1}(z)) \\
&& + \, S^3_{n}(z)-3S_{n}(z)S^2_{n-1}(z)+zS^3_{n-1}(z) \big].
\end{eqnarray*}
By Property 1.4, we have $S^3_{n}(z)-3S_{n}(z)S^2_{n-1}(z)+zS^3_{n-1}(z)=S_{3n}(z).$ Hence $P_{a^{-1}wab^{-1}}-P_{wb^{-1}}=(2-\gamma)Q(x,y,z)$, where
$$Q(x,y,z)=(x^2+y^2)S_{n}(z)S^2_{n-1}(z)-xyS_{n-1}(z)(S^2_{n}(z)+S^2_{n-1}(z))+S_{3n}(z).$$

This proves that character variety of $\fb(6n+2,3)$ is the zero set of the polynomial $(x^2+y^2+z^2-xyz-4)Q(x,y,z)$. To complete the proof of Theorem 2 for $\fb(6n+2,3)$, we only need to show the following.

\begin{proposition}
\label{m3}
$Q(x,y,z)$ is irreducible in $\BC[x,y,z]$.
\end{proposition} 

\begin{proof}
Let $Q'(x,y,z)$ be the polynomial in $\BC[x,y,z]$ defined by $Q'(x,y,z)=Q'(x+y,x-y,z)$. To prove Proposition \ref{m3}, we only need to prove that $Q'$ is irreducible in $\BC[x,y,z]$.

We have
\begin{eqnarray*}
Q' &=& 2(x^2+y^2)S_{n}(z)S^2_{n-1}(z)-(x^2-y^2)S_{n-1}(z)(S^2_{n}(z)+S^2_{n-1}(z))+S_{3n}(z)\\
         &=& y^2S_{n-1}(z) (S_n(z)+S_{n-1}(z))^2-x^2S_{n-1}(z) (S_n(z)-S_{n-1}(z))^2+S_{3n}(z).
\end{eqnarray*}
Since $S_{n-1}(z) (S_n(z)+S_{n-1}(z))^2 \not\equiv 0$ and $\gcd(S_{n-1}(z),S_{3n}(z))=1$ in $\BC[z]$, we have that $Q'$ is irreducible in $\BC[x,y^2]$. Moreover, since $Q' \mid_{x=0,y=0} \, =S_{3n}(z) \not\equiv 0$, $y$ is not a factor of $Q'$ in $\BC[x,y]$.

Assume that $Q'$ is reducible in $\BC[x,y,z]$. Then it must have the form $Q'=(yf+g)(yf-g)$, for some $f,g \in \BC[x,z] \setminus \{0\}$ satisfying $\gcd(f,g)=1$. In particular, we have $f^2=S_{n-1}(z) (S_n(z)+S_{n-1}(z))^2$. This can not occur unless $n=1$. In the case $n=1$, we have $Q'=y^2(z+1)^2-x^2(z-1)^2+z^3-2z$. Hence  $g^2=x^2(z-1)^2-(z^3-2z)$, which can not occur. The proposition follows.
\end{proof}

This complete the proof of Theorem 2 for the two-bridge link $\fb(6n+2,3)$. The proof of Theorem 2 for $\fb(6n+4,3)$ is similar.

\subsection{Twisted Whitehead links} The proof of Theorem 3 for the twisted Whitehead link $W_{2n}=\fb(8n+4,4n+1)$ will be similar to that for $W_{2n-1}=\fb(8n,4n-1)$. Hence, without loss of generality we can assume that $L=\fb(8n,4n-1)$. In this case, we have $\pi_L=\la a,b \mid wa=aw\ra$, where $w=(bab^{-1}a^{-1})^na(a^{-1}b^{-1}ab)^n$. 

We first calculate the character variety of $L$. Recall that $\gamma=P_{aba^{-1}b^{-1}}=x^2+y^2+z^2-xyz-2$. In $SL_2(\BC)$, by applying Eq. \eqref{matrix} we have
\begin{eqnarray*}
awa^{-1}b^{-1} &=& a(bab^{-1}a^{-1})^na(a^{-1}b^{-1}ab)^na^{-1}b^{-1} \\
               &=& S^2_{n-1}(\gamma)abab^{-1}a^{-1}b^{-1}aba^{-1}b^{-1}+S^2_{n-2}(\gamma)ab^{-1}\\
               &&-S_{n-1}(\gamma)S_{n-2}(\gamma)(abab^{-1}a^{-1}b^{-1}+ab^{-1}aba^{-1}b^{-1}).
\end{eqnarray*}
and
\begin{eqnarray*}
wb^{-1} &=& (bab^{-1}a^{-1})^na(a^{-1}b^{-1}ab)^nb^{-1} \\ 
        &=& S^2_{n-1}(\gamma)bab^{-1}a^{-1}b^{-1}a+S^2_{n-2}(\gamma)ab^{-1}-S_{n-1}(\gamma)S_{n-2}(\gamma)(bab^{-2}+b^{-1}a).
\end{eqnarray*}
Hence
\begin{eqnarray}
\label{W}
awa^{-1}b^{-1} - wb^{-1} &=& S^2_{n-1}(\gamma)(abab^{-1}a^{-1}b^{-1}aba^{-1}b^{-1}-bab^{-1}a^{-1}b^{-1}a) \nonumber\\
&&- \, S_{n-1}(\gamma)S_{n-2}(\gamma)(abab^{-1}a^{-1}b^{-1}+ab^{-1}aba^{-1}b^{-1}-bab^{-2}-b^{-1}a). 
\end{eqnarray}

By applying the algorithm for calculating $P_u$, for any word $u$ in 2 letters $a$ and $b$, described in \cite[prop. 1.4.1]{CS}, we can prove the following.

\begin{lemma}
\label{some-traces}
One has
\begin{eqnarray*}
P_{abab^{-1}a^{-1}b^{-1}} &=& xy-(x^2+y^2-3)z+xyz^2-z^3,\\
P_{ab^{-1}aba^{-1}b^{-1}} &=& xy(x^2 + y^2-3)-(x^2y^2+x^2+y^2-3)z+2xyz^2-z^3,\\
P_{abab^{-1}a^{-1}b^{-1}aba^{-1}b^{-1}} &=& xy(x^2 + y^2-3)-(x^4+y^4+ 3 x^2 y^2- 5 x^2  - 5 y^2+5) z\\
   && + \, 2 x y (x^2 + y^2-2) z^2-(x^2 y^2 + 2 x^2 + 2 y^2 -5 ) z^3+2 x y z^4 - z^5.
\end{eqnarray*}
\end{lemma}

From Lemma \ref{some-traces} we have
\begin{eqnarray*}
P_{abab^{-1}a^{-1}b^{-1}aba^{-1}b^{-1}}-P_{bab^{-1}a^{-1}b^{-1}a} &=& (2-\gamma)( xy-\gamma z),\\
P_{abab^{-1}a^{-1}b^{-1}}+P_{ab^{-1}aba^{-1}b^{-1}}-P_{bab^{-2}}-P_{b^{-1}a} &=& (2-\gamma)(xy-2z).
\end{eqnarray*}
These, together with Eq. \eqref{W} imply that
$$
P_{awa^{-1}b^{-1}} - P_{wb^{-1}} = (2-\gamma) S_{n-1}(\gamma) \left[ ( xy-\gamma z)S_{n-1}(\gamma)-(xy-2z)S_{n-2}(\gamma) \right].
$$

Let $Q(x,y,z)=( xy-\gamma z)S_{n-1}(\gamma)-(xy-2z)S_{n-2}(\gamma)$. Then character variety of $\fb(8n,4n-1)$ is the zero set of the polynomial $(\gamma-2) S_{n-1}(\gamma) Q(x,y,z)$.

By Property 1.2, we have $$S_{n-1}(\gamma)=\prod_{j=1}^{n-1}(\gamma-2\cos\frac{j\pi}{n})=\prod_{j=1}^{n-1}(x^2+y^2+z^2-xyz-2-2\cos\frac{j\pi}{n}).$$ Note that, for any $\delta \in \BC$, the polynomial $x^2+y^2+z^2-xyz+\delta$ is irreducible in $\BC[x,y,z]$. 

\no{
\begin{lemma}
$x^2+y^2+z^2-xyz-2-2\cos\frac{j\pi}{n}$ does not determine the canonical component.
\end{lemma}

\begin{proof}
For $x=y=2$, we have $z=2 \pm 2 i \sin \frac{j\pi}{2n}$. For a discrete faithful representation, Shimizu's lemma says that $|z-2| \ge 1$.
\end{proof}
}

To complete the proof of Theorem 3 for $\fb(8n,4n-1)$, we only need to show the following.

\begin{proposition}
\label{W-irred}
$Q(x,y,z)$ is irreducible in $\BC[x,y,z]$.
\end{proposition} 

\begin{proof}

We have $Q=zx^{2n}+Q'$, for some $Q' \in \BC[x,y,z]$ satisfying $\deg_x(Q')<2n$. Since $Q \mid_{z=0} \, = xy(S_{n-1}(x^2+y^2-2)-S_{n-2}(x^2+y^2-2)) \not\equiv 0$, $z$ is not a factor of $Q$.

Assume that $Q$ is reducible in $\BC[x,y,z]$. Then there exist $1 \le k \le 2n-1$ and $f,g \in \BC[x,y,z]$ such that $\deg_x(f)<k$, $\deg_x(g)<2n-k$ and $Q=(x^k+f)(zx^{2n-k}+g)$. It follows that, for any $z_0 \not= 0$, $Q \mid_{z=z_0}$ is reducible in $\BC[x,y]$.

We have $Q=xy(S_{n-1}(\gamma)-S_{n-2}(\gamma))+z(2 S_{n-2}(\gamma) - \gamma S_{n-1}(\gamma))$. Let $R(x,y)$ be the polynomial in $\BC[x,y]$ defined by $R(x,y)=Q(x+y,x-y,2)$. Then 
\begin{eqnarray*}
R &=& (x^2-y^2)(S_{n-1}(\delta)-S_{n-2}(\delta))+4S_{n-2}(\delta)-2\delta S_{n-1}(\delta)\\
  &=& (x^2-y^2)(S_{n-1}(\delta)-S_{n-2}(\delta))-2(S_{n}(\delta)-S_{n-2}(\delta)),
\end{eqnarray*}
where $\delta=4y^2+2$. Since $S_{n-1}(\delta)-S_{n-2}(\delta) \not\equiv 0$ and $\gcd(S_{n-1}(\delta)-S_{n-2}(\delta), S_{n}(\delta)-S_{n-2}(\delta))=1$ in $\BC[y]$, we have that $R$ is irreducible in $\BC[x^2,y]$. Moreover, since $R \mid_{x=0,y=0} \, =-4 \not= 0$, $x$ is not a factor of $R$ in $\BC[x,y]$.

Since $Q \mid_{z=2}$ is reducible in $\BC[x,y]$, so is $R$. Hence $R$ must have the form $R=(xf+g)(xf-g)$, for some $f,g \in \BC[y] \setminus \{0\}$ satisfying $\gcd(f,g)=1$. In particular, we have 
\begin{eqnarray*}
f^2 &=& S_{n-1}(\delta)-S_{n-2}(\delta)=\prod_{j=1}^{n-1} (\delta-2\cos \frac{(2j-1)\pi}{2n-1})\\
    &=& 4^{n-1} \prod_{j=1}^{n-1} (y+i\sin \frac{(2j-1)\pi}{4n-2})(y-i\sin \frac{(2j-1)\pi}{4n-2}).
\end{eqnarray*}
This can not occur unless $n=1$. In the case $n=1$, $R=x^2-(9y^2+4)$ is irreducible in $\BC[x,y]$. The proposition then follows.
\end{proof}

This complete the proof of Theorem 3 for the two-bridge link $\fb(8n,4n-1)$. The proof of Theorem 3 for $\fb(8n+4,4n+1)$ is similar.

\no{
In this case $w=(bab^{-1}a^{-1})^nbab(a^{-1}b^{-1}ab)^n$ and $\pi=\la a,w \mid wa=aw \ra$. The character variety is the zero set of $P_{wb^{-1}}-P_{awa^{-1}b^{-1}}.$

We write $(bab^{-1}a^{-1})^n=S_n(t)I-S_{n-1}(t)(bab^{-1}a^{-1})^{-1}$ and $(bab^{-1}a^{-1})^n=S_n(t)I-S_{n-1}(t)(a^{-1}b^{-1}ab)^{-1}$.

We have
\begin{eqnarray*}
P_{(bab^{-1}a^{-1})^nbab(a^{-1}b^{-1}ab)^nb^{-1}}&=& P_{babb^{-1}}S^2_{n}(t)+P_{(bab^{-1}a^{-1})^{-1}bab(a^{-1}b^{-1}ab)^{-1}b^{-1}}S^2_{n-1}(t)\\
                   &&-(P_{(bab^{-1}a^{-1})^{-1}babb^{-1}}+P_{bab(a^{-1}b^{-1}ab)^{-1}b^{-1}})S_{n}(t)S_{n-1}(t)\\
                   &=& P_{ba}S^2_{n}(t)+P_{aba^{-1}bab^{-1}}S^2_{n-1}(t)-(P_{ab}+P_{ba})S_{n}(t)S_{n-1}(t).
\end{eqnarray*}
Similarly
\begin{eqnarray*}
P_{a(bab^{-1}a^{-1})^nbab(a^{-1}b^{-1}ab)^na^{-1}b^{-1}}
                   &=&P_{ababa^{-1}b^{-1}}S^2_{n}(t)+P_{ab}S^2_{n-1}(t)-(P_{a^2b^2a^{-1}b^{-1}}+P_{ab})S_{n-1}(t).
\end{eqnarray*}
Hence
\begin{eqnarray*}
P_{waa^{-1}b^{-1}}-P_{awa^{-1}b^{-1}} &=& (P_{ba}-P_{ababa^{-1}b^{-1}})S^2_{n}(t)+(P_{aba^{-1}bab^{-1}}-P_{ab})S^2_{n-1}(t)\\
&&-(P_{a^2b^2a^{-1}b^{-1}}-P_{ba})S_{n}(t)S_{n-1}(t)\\
&=&(2-t)(S_{n}(t)-S_{n-1}(t)) \left[ zS_{n}(t)-(xy-z)S_{n-1}(t) \right].
\end{eqnarray*}
}

\end{document}